\documentclass[11pt,oneside]{amsart}

\usepackage{graphicx}
\usepackage{amsfonts}
\usepackage{epsf}
\usepackage{amssymb}
\usepackage{amsmath}
\usepackage{amscd}
\usepackage{tikz}
\usepackage{pdfpages}
\usepackage{fancyhdr}
\usepackage{setspace}
\usepackage{hyperref}
\usepackage[all]{xy}
\usetikzlibrary{matrix}
\usepackage{verbatim}
\usepackage{enumerate}

\theoremstyle{theorem}
\newtheorem{theorem}{Theorem}[section]
\newtheorem{proposition}[theorem]{Proposition}
\newtheorem{lemma}[theorem]{Lemma}

\newtheorem{corollary}[theorem]{Corollary}
\newtheorem{conjecture}[theorem]{Conjecture}

\newtheorem*{CC}{Cabling Conjecture}
\newtheorem*{TS}{Two Summands Conjecture}

\newcommand{\Z}{\mathbb{Z}}

\newcommand{\Spinc}{\text{Spin}^c}

\newcommand{\spinc}{\text{spin}^c}

\newcommand{\D}{\mathcal{D}}

\newcommand{\be}{\begin{enumerate}}
\newcommand{\ee}{\end{enumerate}}	
\newcommand{\bi}{\begin{itemize}}
\newcommand{\ei}{\end{itemize}}

\makeatletter
\newtheorem*{rep@theorem}{\rep@title}
\newcommand{\newreptheorem}[2]{%
\newenvironment{rep#1}[1]{%
 \def\rep@title{#2 \ref{##1}}%
 \begin{rep@theorem}}%
 {\end{rep@theorem}}}
\makeatother

\newreptheorem{theorem}{Theorem}
\newreptheorem{lemma}{Lemma}
\newreptheorem{question}{Question}
\newreptheorem{corollary}{Corollary}

\begin{document}

\rhead{\thepage}
\lhead{\author}
\thispagestyle{empty}


\raggedbottom
\pagenumbering{arabic}
\setcounter{section}{0}


\title{A note on cabled slice knots and reducible surgeries}

\author{Jeffrey Meier}
\address{Department of Mathematics, Indiana University, 
Bloomington, IN 47408}
\email{jlmeier@indiana.edu}
\urladdr{http://pages.iu.edu/~jlmeier/} 

\begin{abstract}
	We consider the question of when a slice knot admits a reducible Dehn surgery.  By analyzing the correction terms associated to such a surgery, we show that slice knots cannot admit surgeries with more than two summands. We also give a necessary Heegaard Floer theoretic condition for a positive cable of a knot to be slice.
\end{abstract}

\maketitle

\section{Introduction}

Dehn surgery is one of the simplest and most important operations in three-manifold topology, and understanding which three-manifolds can result from Dehn surgery on a knot in $S^3$ has been one of the major goals of modern low-dimensional topology.  We refer the reader to \cite{gordon:Dehn_survey} for an introduction and comprehensive overview of the history and scope of this endeavor.

Perhaps the most basic question related to this goal is when the result of Dehn surgery can be a reducible manifold (a manifold containing an essential two-sphere). Moser's classification of Dehn surgeries on torus knots \cite{moser:torus} gave the first nontrivial examples of this phenomenon:
$$S^3_{pq}(T_{p,q}) \cong L(p,q)\#L(q,p).$$
Of course, the fact that $S^3_0(U)\cong S^1\times S^2$ can be thought of as a degenerate case of this.  Gabai showed \cite{gabai:surgery} that no other knot in $S^3$ admits a surgery to $S^1\times S^2$, so we may assume that any nontrivial, reducible surgery decomposes as a connected sum.

Other interesting examples of reducible surgeries come from considering cabled knots \cite{gordon:satellite}.  Let $J_{p,q}$ denote the $(p,q)$--cable of $J$, for some pair $p,q\in\Z$ with $p\geq 2$. (Throughout, $p$ is the longitudinal winding number.)  Then we have
$$S^3_{pq}(J_{p,q})\cong L(p,q)\# S^3_{q/p}(J).$$
Cabled knots represent the only known examples of knots admitting reducible surgeries (thinking of $T_{p,q}$ as a cable of the unknot), and we have the following conjecture of Gonzales Acu\~na and Short.

\begin{CC}[\cite{GAS}]
	If $K$ is a nontrivial knot in $S^3$ and $S^3_r(K)$ is reducible, then $K=J_{p,q}$ for some knot $J$, and $r=pq$.
\end{CC}

There has been much progress made towards a positive resolution of this conjecture.  For example, it is known that any reducible surgery slope must be integral \cite{gordon-luecke:complements} and that one summand must be a lens space \cite{gordon-luecke:summand}. Moreover, if a reducible surgery on $K$ yields a connected sum of lens spaces, then $K$ is either a torus knot or a cable thereof \cite{greene:cabling}.  Many classes of knots are known to satisfy the the conjecture, including knots with symmetries \cite{EM:cabling,hayashi-motegi,hayashi-shimokawa}, alternating and arborescent knots \cite{men-thistle:reducible,wu:arborescent}, and knots with low bridge number \cite{grove,hoffman:great,howie:bridge,sayari, zufelt}.

Since any reducible surgery slope must be integral, it follows that every reducible surgery on a cable knot yields a connected sum of a lens space with an \emph{irreducible} rational homology sphere.  In particular, all such surgeries have only two summands.  This motivates the following conjecture.

\begin{TS}
	If $K$ is a nontrivial knot in $S^3$ and $S^3_r(K)$ is reducible, then $S^3_r(K)\cong Y_1\#Y_2$, with $Y_1$ and $Y_2$ irreducible.
\end{TS}

In a reducible surgery, at most one of the summands is an integer homology sphere and at most two of the summands are lens spaces \cite{valdez,howie:groups}.  It follows that a reducible surgery with three summands (the highest possible number) must consist of two lens spaces summands and an integer homology sphere summand.  The Two-Summands Conjecture is true for knots with bridge number at most five and positive braid closures \cite{zufelt}.

In the present paper, we restrict our attention to slice knots.  A knot $K\subset S^3$ is called \emph{slice} if there exists a smooth, properly embedded disk in $D\subset B^4$ with $\partial D=K$.  Our first result verifies the Two Summands Conjecture for slice knots.

\begin{theorem}\label{thm:twosummand}
	A slice knot in the three-sphere cannot admit a reducible surgery with three irreducible summands.
\end{theorem}

The proof is straightforward.  If $K$ is a slice knot, then $S^3_{pq}(K)$ is integer homology cobordant to $L(pq,1)$.  Knowing this, one can compare the correction terms of $L(pq,1)$ to those of $L(p,a)\#L(q,b)$.  A simple lemma shows that these collections never match up.  In fact, Theorem \ref{thm:twosummand} holds for any knot $K$ with $V_0(K)=V_0(\overline K)=0$, where $V_0(K)$ is a Heegaard Floer theoretic knot invariant coming from the knot Floer complex that determines the correction terms of surgeries on $K$ \cite{oz-sz:integer,oz-sz:rational}, and $\overline K$ denotes the mirror of $K$.  The condition that $V_0(K)=0$ suffices in the case that the surgery is positive.  Along these same lines, we have the following observation.

\begin{theorem}\label{thm:cable}
	Suppose that $K$ is a positive cable of a knot $J$.  If $K$ is slice, then $V_0(J)=0$.
\end{theorem}

As an application, let $D$ denote the untwisted Whitehead double of the right-handed trefoil knot. By Proposition 6.1 of \cite{HKL}, we know that $V_0(D)=1$. It follows that no positive cable of $D$ is smoothly slice, even though the $(p,1)$--cable of $D$ is topologically slice for all $p$.  (Note that this particular fact can also be deduced from work of Hom \cite{hom:cable}.) We're led to the following natural conjecture.

\begin{conjecture}\label{conj:cabled}
	Suppose that $K$ is the $(p,q)$--cable of $J$.  Then $K$ is slice if and only if $J$ is slice and $q=1$.
\end{conjecture}

Proposition \ref{prop:alex} below shows that if $J_{p,q}$ is algebraically slice, then $q=1$. Since any $(p,1)$--cable of slice knot is also slice, one direction of the conjecture is true.  Theorem \ref{thm:cable} and \cite{hom:cable} give evidence that the conjecture is true at the level of Heegaard Floer homology. Note that Conjecture \ref{conj:cabled} is true for fibered knots in the homotopy-ribbon setting by Theorem 8.5 of \cite{miyazaki:fibered}.
 
\section*{Acknowledgements}

The author wishes to thank Nicholas Zufelt, whose exposition of the Cabling Conjecture during a talk at Indiana University inspired this note and whose comments improved the manuscript.  The author is also grateful to Tye Lidman, Charles Livingston, and Margaret Doig, whose insights helped to improve the paper to its present form.  This work was supported by the National Science Foundation under grant DMS-1400543.

\section{Heegaard Floer correction terms}

The main tool used in proving Theorem \ref{thm:twosummand} is the correction terms coming from Heegaard Floer homology.  This theory was first formulated for closed three-manifolds \cite{oz-sz:3-manifolds_1}, before being shown to give invariants of four-manifolds \cite{oz-sz:absolute}.  For our purposes, the most important aspects of the theory will be the Heegaard Floer correction terms, defined in \cite{oz-sz:absolute}.

Let $Y$ be a oriented, closed three-manifold, and let $\Spinc(Y)$ denote the collection of $\Spinc$ structures associated to $Y$.  For each $\frak s\in\Spinc(Y)$, let $d(Y,\frak s)$ denote the \emph{correction term} associated to $(Y,\frak s)$.  A detailed development of this invariant can be found in \cite{oz-sz:absolute}, where it was shown to have the following properties.
\be
	\item Let $-Y$ denote the opposite orientation of $Y$. Then  $d(-Y,\frak s)=-d(Y,\frak s)$.
	\item Let $\overline{\frak s}$ be the image of $\frak s$ under conjugation.  Then, $d(Y,\overline{\frak s}) = d(Y,\frak s)$.
	\item For any pairs $(Y_1,\frak s_1)$ and $(Y_2,\frak s_2)$,
	$$d(Y_1\#Y_2,\frak s_1\#\frak s_2) = d(Y_1,\frak s_1)+d(Y_2,\frak s_2).$$
\ee



Let $\D(Y)$ denote the collection of correction terms associated to $Y$. (Note that elements of this set can appear with multiplicity greater than one.) Recall that there is an affine identification $\Spinc(Y)\approx H^2(Y;\Z)$.  For our purposes, the most important aspect of the correction terms is that they are preserved under integer homology cobordism.

\begin{proposition}\label{prop:cobordant}
	If $Y_1$ and $Y_2$ are integer homology cobordant, then $\mathcal D(Y_1)=\mathcal D(Y_2)$.
\end{proposition}

The correspondence between the collections of correction terms in the above proposition can be strengthened; see \cite{doig-wehrli} for a proof and more detail.  Here, we will not be too concerned with fixing identifications and labelings of $\spinc$ structures.

Let $L(p,q)$ denote the lens space obtained by $p/q$--surgery on the unknot in $S^3$.  In this case, Ozv\'ath and Szab\'o \cite{oz-sz:absolute} gave a canonical ordering on $\Spinc(L(p,q))$ by elements $i\in\Z_p$, as well as a recursive formula for the correction terms.

\begin{proposition}\label{prop:lensterms}
	For any positive, relatively prime integers $p>q$ and any integer $0\leq i<p+q$, we have
\begin{equation}\label{eqn:lens}
d(-L(p,q),i)=\frac{pq-(2i+1-p-q)^2}{4pq}-d(-L(q,r),j),
\end{equation}
	where $r$ and $j$ are the reductions modulo $q$ of $p$ and $i$, respectively.
\end{proposition}

This formula allows for the calculation the correction terms of any lens space, since $L(1,0)\cong S^3$, and $d(S^3)=0$. 

Now, if $K$ is slice, then there exists a concordance from $K$ to the unknot $U$.  It follows that $S^3_{p/q}(K)$ is integer homology cobordant to $S^3_{p/q}(U)\cong L(p,q)$.  (The Dehn surgery manifolds cobound a four-manifold that is obtained by performing ``Dehn surgery cross $I$'' on the concordance.  Thus, we obtain the following corollary to Proposition \ref{prop:cobordant}.

\begin{corollary}\label{coro:slice}
	If $K$ is slice, then 
	$$\mathcal D(S^3_{p/q}(K))=\mathcal D(L(p,q)).$$
\end{corollary}

In fact, the above corollary holds for any $K$ with $V_0(K)=0$.  This follows from the integer and rational surgery formulae developed in \cite{oz-sz:integer,oz-sz:rational}.

\section{Proof of Theorem \ref{thm:twosummand}}\label{sec:proof}

In this section, we prove the main theorem. First, we present a simple lemma that gives an upper bound on the range of the correction terms for a given lens space.  Let $\Delta(p,q)=\max\D(L(p,q))-\min\D(L(p,q))$ denote this range.

\begin{lemma}\label{lemma:range}
	$\Delta(p,q)\leq p/4$
\end{lemma}

\begin{proof}
	Note that $\Delta(p,q)=\Delta(p,p-q)$, since $L(p,q)\cong-L(p,p-q)$. Therefore, we may assume without loss of generality that $p>q/2>0$.  Furthermore, the recursive formula in Proposition \ref{prop:lensterms} is easy to understand for small values of $q$, and it can be checked that, $\Delta(p,q)\leq p/4$ whenever $q<8$.  (The calculations breaks down into a finite number of cases based on the values of $q$ and $r$.)
	
	Thus, we assume that $q\geq 8$.  From Equation \ref{eqn:lens}, we have
	$$ \Delta(p,q) \leq \max_{i,i'}\left\{\frac{(2i-p-q+1)^2-pq}{4pq}-\frac{(2i'-p-q+1)^2-pq}{4pq}\right\}+\Delta(q,r).$$
	The bracketed term is maximized when $i=0$ and $i'=(p+q-1)/2$ or $i'=(p+q)/2$, depending on whether $p$ and $q$ have opposite parity or not.  The former choice of $i'$ yields a more extremal value; so, in either case, we have
\begin{eqnarray*}
	\Delta(p,q) & \leq & \frac{(p+q-1)^2-pq}{4pq}-\frac{-pq}{4pq}+\Delta(q,r) \\
	& = & \frac{(p+q-1)^2}{4pq}+\Delta(q,r).
\end{eqnarray*} 
	Since $q<p$, we have $p+q-1<2p$, so
\begin{eqnarray*}
	\Delta(p,q) & \leq & \frac{4p^2}{4pq}+\Delta(q,r) \\
	& = & \frac{p}{q}+\Delta(q,r).
\end{eqnarray*} 
	Now, having shown that the claim holds for small values of $p$, we can proceed by induction and assume that $\Delta(q,r)\leq q/4$.  Since $q\geq 8$ and $p>q/2$, we have
$$ 	\Delta(p,q)  \leq  \frac{p}{8}+\frac{q}{4} 
	 \leq  \frac{p}{8}+\frac{p}{8} 
	 = \frac{p}{4}.
$$
\end{proof}

We remark that this bound is sharp, since $\Delta(p,1)$ is $p/4$ when $p$ is even and $(p^2-1)/4p$ when $p$ is odd.  However, this bound appears to be far from sharp once $q>1$.

\begin{reptheorem}{thm:twosummand}
	A slice knot in the three-sphere cannot admit a reducible surgery with three irreducible summands.
\end{reptheorem}

\begin{proof}
	Suppose that $K$ is a slice knot such that $S^3_r(K)$ has three irreducible summands. Since $S^3_{-r}(\overline K)\cong -S^3_r(K)$, there will be no loss of generality in assuming $r>0$.  Recall from the introduction that we must have at most two lens space summands and at most one non-lens space summand, which must be an integer homology sphere:
	$$S^3_{pq}(K)=L(p,a)\#L(q,b)\#Y.$$
	It follows that $r=pq$, with $p$ and $q$ coprime.  Since $Y$ is an integer homology sphere, it has only one correction term, $d(Y)$.  By Corollary \ref{coro:slice}, we know that $\mathcal D(S^3_r(K))=\mathcal D(L(r,1))$.  These two facts give us the following:
	\begin{equation}
		\mathcal D(L(pq,1))=\mathcal D(L(p,a)\#L(q,b))+d(Y).
	\end{equation}
	It follows that the ranges of values on the left and the right must match. In particular, since $d(Y)$ is constant, we must have 
	$$\Delta(pq,1)\leq\Delta(p,a)+\Delta(q,b).$$
	By Lemma \ref{lemma:range}, we have that $\Delta(p,a)\leq p/4$ and $\Delta(q,b)\leq q/4$, so we have
	\begin{equation}\label{eqn:ineq}
	\Delta(pq,1)\leq \frac{p+q}{4}.	
	\end{equation}
	Now, if $pq$ is even, we have $\Delta(pq,1)=pq/4$; so $pq\leq p+q$. If $pq$ is odd, we have $\Delta(pq,1)=(pq-1)(pq+1)/4pq$; so $pq-1\leq p+q$. It follows that either $p=1$, $q=1$, or $p=q=2$.  However, we have assumed that $p$ and $q$ are coprime.  Therefore, either $L(p,a)$ or $L(q,b)$ is $S^3$, and the proof is complete.
\end{proof}

\section{Proof of Theorem \ref{thm:cable}}\label{sec:proof2}

We are grateful to Chuck Livingston for pointing out that the following result should hold in the algebraically slice setting. Not only is this fact interesting in its own right, but it greatly simplifies the proof of the ensuing theorem, which was originally entirely Heegaard Floer theoretic.

\begin{proposition}\label{prop:alex}
	Let $K = J_{p,q}$ denote the $(p,q)$--cable of a knot $J$.  If $K$ is algebraically slice, then $q=1$.
\end{proposition}

\begin{proof}
	Suppose that $K=J_{p,q}$ is algebraically slice.  Since $K$ is algebraically slice, we have that $\Delta_K(t) = f(t)f(t^{-1})$ for some $f(t)\in\Z[t]$. Furthermore, since $K$ is a cable we can write $\Delta_K(t) = \Delta_J(t^p)\Delta_{T_{p,q}}(t)$.  It follows that
	$$\Delta_J(t^p)\Delta_{T_{p,q}}(t) = f(t)f(t^{-1}).$$
	
	Let $\xi$ be a root of $\Delta_{T_{p,q}}(t)$.  This means that $\xi$ is a $pq^\text{th}$--root of unity that is neither a $p^\text{th}$--root of unity nor a $q^\text{th}$--root of unity.  It follows that $\xi$ is a root of $f(t)f(t^{-1})$.  Without loss of generality, we can assume that $f(\xi)=0$.  It follows that $f(\overline\xi)=0$ as well, and we note that $\overline\xi=\xi^{-1}$.  It follows that $\xi$ is a root of both $f(t)$ and $f(t^{-1})$. Since $\xi$ has multiplicity one as a root of $\Delta_{T_{p,q}}(t)$, and since we have observed that it has multiplicity two as a root of $f(t)f(t^{-1})$, it follows that $\xi$ is a root of $\Delta_J(t^p)$.  Thus, $\xi^p$ is a root of $\Delta_J(t)$.  This is true of any root $\xi$ of $\Delta_{T_{p,q}}$.
	
	Now, let $q_1$ be a prime factor of $q$, and let $\xi=e^{2\pi i/pq_1}$.  It follows that $\xi^{pq}=1$, but $\xi^q\not=1$ and $\xi^p\not=1$, so $\xi$ is a root of $\Delta_{T_{p,q}}(t)$.  Therefore, $\xi^p$ is a root of $\Delta_J(t)$. However, $\xi^p$ is a primitive $q_1^\text{th}$--root of unity and $q_1$ is prime.  This implies that the cyclotomic polynomial $\Phi_{q_1}(t)$ divides $\Delta_J(t)$, which in implies that $\Phi_{q_1}(1)=q_1$ divides $\Delta_J(1)$, which is one, since $\Delta_J(t)$ is an Alexander polynomial. It follows that $q_1=1$ and that this conclusion must hold for all prime factors of $q$.  Therefore, we must have $q=1$, as desired.
\end{proof}

\begin{reptheorem}{thm:cable}
	Suppose that $K$ is a positive cable of a knot $J$. If $K$ is slice, then $V_0(J)=0$.
\end{reptheorem}

\begin{proof}
	If $K$ is the $(p,q)$--cable of $J$, and $K$ is algebraically slice, then $q=1$ by Proposition~\ref{prop:alex}.  It follows that
	$$S^3_p(K)=L(p,1)\#S^3_{1/p}(J).$$
	If $K$ is slice, then Corollary \ref{coro:slice} tells us that
	$$\mathcal D(L(p,1)) = \mathcal D(L(p,1))+d(S^3_{1/p}(J)).$$
	It follows that $d(S^3_{1/p}(J))=0$, which is equivalent to $V_0(J)=0$ \cite{oz-sz:integer,oz-sz:rational}.
\end{proof}


\bibliographystyle{abbrv-fr}
\bibliography{CabledSlice}

\end{document}